\documentclass[12pt]{amsart}
\usepackage[utf8]{inputenc}
\usepackage[margin=1in]{geometry}

\usepackage{changes}
\usepackage{comment}
\usepackage{accents}

\usepackage{fullpage}
\usepackage{amssymb,amsfonts}
\usepackage[all,arc]{xy}
\usepackage{enumerate}
\usepackage{enumitem}
\usepackage{mathrsfs}
\usepackage{mathtools}
\usepackage{tikz}
\usepackage{tikz-cd}
\usetikzlibrary{patterns}
\usepackage{graphicx}
\usepackage[font={small}]{caption}
\usepackage{float}
\usepackage{subcaption}
\usepackage{hyperref}
\usepackage{rotating}
\usepackage{changes}
\usepackage{standalone}
\usepackage{changes}
\usepackage{relsize}
\usepackage{thm-restate}
\usepackage{extpfeil}

\newtheorem{thm}{Theorem}[section]
\newtheorem*{thm*}{Theorem}

\newtheorem{prop}[thm]{Proposition}
\newtheorem{lem}[thm]{Lemma}

\newtheorem{quest}[thm]{Question}

\theoremstyle{definition}
\newtheorem{defn}[thm]{Definition}

\newtheorem{exmp}[thm]{Example}

\theoremstyle{remark}

\newcommand{\Z}{\mathbb{Z}}
\newcommand{\Q}{\mathbb{Q}}
\newcommand{\R}{\mathbb{R}}
\newcommand{\C}{\mathbb{C}}
\newcommand{\F}{\mathbb{F}}
\newcommand{\N}{\mathbb{N}}

\title{Relative knot concordance}

\begin{document}

\begin{abstract}
Motivated by the relative-slice problem in $4$-manifold topology, we define for a knot $R$ the notion of a relative knot concordance group $\mathcal{C}_R$ and analyze the structure of its algebraic analogue $\mathcal{AC}_R$ modulo torsion. Using Milnor signatures, we show that $\mathcal{AC}_R$ has infinite rank for every $R$. We also show that any prescribed finite set of natural numbers can be realized as orders of elements in a single relative concordance group.
\end{abstract} 

\author{Vyacheslav Krushkal}
\address{Department of Mathematics, University of Virginia, Charlottesville, VA 22904}
\email{krushkal\char 64 virginia.edu}

\author{Ryan Stees}
\address{Department of Mathematics, The University of Texas at Austin, Austin, TX 78712}
\email{ryan.stees\char 64 austin.utexas.edu}
\maketitle


\section{Introduction}

Motivated by the $4$-dimensional topological surgery conjecture \cite{Freedman} for free nonabelian groups, Freedman and Lin \cite{FL} introduced  
the {\em relative-slice problem}, which concerns embeddability of $4$-dimensional $2$-handlebodies into the $4$-ball. Given a Kirby diagram $(L,R)$ of a $4$-manifold, where the disjoint links $R,L$ in $S^3$ describe the $1$- and $2$-handles respectively, the problem asks:  \begin{equation} \label{eq:RL} \text{Is $L$ slice in the $4$-manifold }\; D^4\cup_{R, 0\text{-framing}}\, 2\text{-handles}?\end{equation}
The question \eqref{eq:RL} can be formulated in both smooth and topological categories, and we refer the reader to \cite{Kr, FK} for a detailed discussion and applications of this approach to surgery. 

In this paper we isolate a particular facet of the theory that concerns knots rather than more general links. In this setting we establish a relative version of the classical knot concordance theory, and specifically a relative version ${\mathcal C}_R$ of the knot concordance group $\mathcal{C}$. 

Given a fixed knot type $R$, two knots $K_0, K_1$  with $K_i\subset S^3\times \{ i\}, i=0,1$ are called {\em concordant relative to $R$} if they cobound a (topologically flat or smooth, depending on the category) annulus in $S^3\times [0,1]\cup 0$-framed $2$-handles, attached to a finite union of split copies of $R$ in $S^3\times \{ 0\}$ and $S^3\times \{ 1\}$; see Section~\ref{sec:def} for a more detailed discussion and a proof that this defines an equivalence relation. 
If $K_0, K_1$ are concordant in the classical sense, they are concordant relative to $R$, so we obtain a homomorphism \[\text{rel}_R:\mathcal{C}\twoheadrightarrow\mathcal{C}_R\] with kernel $\mathcal{R}$, the subgroup of concordance classes of knots which are \emph{slice relative to $R$}.

We develop analogues of tools from classical concordance theory, and specifically from algebraic concordance, to analyze relative concordance groups $\mathcal{C}_R$. 
Milnor \cite{Milnor} and J. Levine \cite{Levine} initiated the study of the so-called \emph{abelian} concordance invariants which extract concordance data from the Alexander module of a knot, the first homology of the universal abelian cover of the knot exterior.
There is a surjective homomorphism \[A:\mathcal{C}\twoheadrightarrow\mathcal{AC}\] onto what is now called the \emph{algebraic concordance group}, the group $\mathcal{AC}$ of Witt equivalence classes of Seifert forms of knots in $S^3$, which Levine showed is isomorphic to \[\Z^\infty\oplus(\Z/2)^\infty\oplus(\Z/4)^\infty.\]

In Section~\ref{sec:Seifert}, we define for a fixed knot $R\subset S^3$ a group $\mathcal{AC}_R$, the \emph{algebraic concordance group relative to $R$}, an analogue of $\mathcal{AC}$ in the context of concordance relative to $R$.
This definition, along with the observation that algebraically slice knots are \emph{algebraically relatively slice}, implies there is a surjective homomorphism \[\mathcal{A}\text{rel}_R:\mathcal{AC}\twoheadrightarrow\mathcal{AC}_R.\]
We then prove that the surjection $A:\mathcal{C}\twoheadrightarrow\mathcal{AC}$ induces a surjection \[A_R:\mathcal{C}_R\twoheadrightarrow\mathcal{AC}_R\] as in the diagram \eqref{diagram:exact}; see Proposition~\ref{prop:halfdies}.

By the existence of $A_R$, the kernel of $\mathcal{A}\text{rel}_R$ contains $A(\mathcal{R})$, the algebraic concordance classes of relatively slice knots.
In Section~\ref{sec:cables}, we show that this kernel equals $A(\mathcal{R})$.
In other words, we show that $\mathcal{AC}_R$, initially defined as a group of equivalence classes of Seifert forms, is the quotient of $\mathcal{AC}$ by the algebraic concordance classes of relatively slice knots.
We accomplish this through a useful characterization of $\mathcal{AC}_R$ as the quotient of $\mathcal{AC}$ by the algebraic concordance classes of the $(n,1)$-cables of the fixed knot $R$ and its orientation-reversed mirror image $-R$, which we show are all relatively slice.
See Proposition~\ref{prop:quotient} for details.
\begin{equation*}
\begin{tikzcd}
0 \arrow[r] & \mathcal{R} \arrow[r,hookrightarrow] \arrow[d,twoheadrightarrow,"A|_\mathcal{R}"'] & \mathcal{C} \arrow[r,twoheadrightarrow,"{\text{rel}}_R"] \arrow[d,twoheadrightarrow,"A"']& \mathcal{C}_R \arrow[d,twoheadrightarrow,"A_R"] \arrow[r] & 0 \\
0 \arrow[r] & A(\mathcal{R}) \arrow[r,hookrightarrow] & \mathcal{AC} \arrow[r,twoheadrightarrow,"\mathcal{A}{\text{rel}}_R"'] & \mathcal{AC}_R \arrow[r] & 0
\end{tikzcd} \tag{$\ast$} \label{diagram:exact}
\end{equation*}

In Section~\ref{sec:signatures}, we prove the following main theorem.

\begin{restatable*}{thm}{summand} \label{thm:main}
For any knot $R\subset S^3$, the algebraic relative concordance group $\mathcal{AC}_R$ admits a surjective homomorphism $\mathcal{AC}_R\twoheadrightarrow\Z^\infty$.
\end{restatable*}

\noindent More precisely, we show that given any $R$, the surjection \[\mathcal{AC}\xtwoheadrightarrow{\bigoplus_{S'}\sigma_\theta}\Z^\infty\] defined by a particular collection of signature invariants $\{\sigma_\theta\}_{\theta\in S'}$ which is independent of $R$ still has infinite rank after passing to algebraic relative concordance.
We have the following commutative square, which relies on the aforementioned Proposition~\ref{prop:quotient}.
\begin{center}
\begin{tikzcd}
\mathcal{AC} \arrow[r,twoheadrightarrow,"\oplus_{S'}\sigma_\theta"] \arrow[d,twoheadrightarrow] & \Z^\infty \arrow[d,twoheadrightarrow] \\
\mathcal{AC}_R \arrow[r,twoheadrightarrow] & \Z^\infty/(\bigoplus_{S'}\sigma_\theta)(A(\mathcal{R})) \arrow[r,twoheadrightarrow] & \Z^\infty
\end{tikzcd} 
\end{center}

Toward this result, we first recall work of Milnor \cite{Milnor} and J. Levine \cite{Levine} on algebraic concordance in Sections~\ref{subsec:Milnorsgn} and~\ref{subsec:Levinesgn}.
Our discussion focuses on signature invariants, which determine the image of any concordance class in the torsion-free part of $\mathcal{AC}$.
We then conclude Section~\ref{sec:signatures} by proving Theorem~\ref{thm:main}, which relies on Milnor's work.

In Section~\ref{sec:torsion}, we briefly discuss torsion in relative concordance groups.
Contrary to the setting of the classical knot concordance group, in which the existence of torsion of order greater than 2 remains a major open question (cf. the K3 problem \cite[Problem 1.38(a)]{K3}), we show that one can readily produce torsion elements in relative concordance groups $\mathcal{C}_R$ for suitable choices of $R$.

The focus in this paper is on relative algebraic concordance, however we expect that many tools in both topological and smooth concordance theories can be adapted to study more subtle phenomena in relative knot concordance.

Another possible version of relative knot concordance involves a knot type $R$ in a solid torus $T$ rather than in $S^3$. In this setting, two knots $K_0,K_1$ are relatively slice if they cobound an annulus in $S^3\times [0,1]\cup 2$-handles, where the zero-framed $2$-handles are attached to $S^3\times \{ 0\}$ and $S^3\times \{ 1\}$ along finitely many copies $R_i\subset T_i$ where the solid tori $T_i$ are regular neighborhoods of meridional circles of $K_0, K_1$. This theory, which has a rather different flavor, is not pursued in the present paper.


\subsection*{Acknowledgements} 

VK was supported in part by NSF grant DMS-2405044.
RS thanks the Department of Mathematics at the University of Virginia---where most of the work for this project took place---for providing a supportive research environment.


\subsection*{AI Disclosure} The authors discussed number-theoretic Proposition~\ref{prop: finite} with ChatGPT 5.6 Sol, which suggested a proof of this statement. The authors checked this proof, and did not use AI in any other parts of the paper other than for proofreading the final manuscript. 


\section{Relative concordance} 
Fix a knot type $R\subset S^3$. Given $r\geq 0$, write $R_r$ to denote $r$ split copies of $R$, so the copies of $R$ in $R_r$ are contained in a collection of $r$ disjoint balls in $S^3$. Denote the union of these balls by $B_r$.
In the following definition, we assume that the knots $J,K$ lie in the complement of the balls containing split copies of $R$.

\begin{defn} \label{def:relconc}
Fix a knot $R\subset S^3$. We say knots $J,K\subset S^3$ are \emph{topologically (resp. smoothly) concordant relative to $R$}, or \emph{relatively concordant}, if for some $r,s \geq 0$ there exists a locally flat (resp. smooth) embedding \[S^1\times[0,1]\hookrightarrow(S^3\times[0,1])\cup(\sqcup_{i=1}^{{r+s}} h_i)\] which restricts on $S^1\times\{0\}$ and $S^1\times\{1\}$ to the knots $J$ and $K$, respectively, where the $h_i$ are 0-framed 2-handles attached to $S^3\times[0,1]$ along 
$R_r\subset S^3\times\{1\}$ and $R_s\subset S^3\times\{0\}$.
\end{defn}

For the following definition, denote by $-R_r$ the union of $r$ split copies of the orientation-reversed mirror image $-R$.

\label{sec:def}
\begin{defn} \label{def:relslice}
We say a knot $K\subset S^3$ is \emph{topologically (resp. smoothly) slice relative to $R$}, or \emph{relatively slice}, if there exist $r,s\geq 0$ such that $K\subset S^3\smallsetminus B_{{r+s}}$
bounds a locally flat (resp. smoothly) embedded disk in $B^4\cup(\sqcup_{i=1}^{{r+s}} h_i)$, where the $h_i$ are 0-framed 2-handles attached to $B^4$ along a split collection $R_r\sqcup -R_s$.
\end{defn}

Suppose $K$ is slice relative to $R$, and let $r(R)$, $m(R)$, and $-R$ denote the orientation-reverse, mirror image, and orientation-reversed mirror image of $R$, respectively.
As a convention, we take $m(R)$ to be the image of $R$ under a reflection $m$ across an equatorial $S^2\subset S^3$ which does not intersect $R$.
We may assume that the slice is transverse to the cocores of the $2$-handles, and further that the intersection of the slice with the $2$-handles consists of a collection of parallel copies of their cores. 
Removing the $(r+s)$ 2-handles attached to $B^4$, we see that $K$ and a link consisting of copies of $R$, $r(R)$, $-R$, and $m(R)$ cobound a connected planar surface in $B^4$.
This link consists of $(r+s)$ sub-links $R_i(n_i^+,n_i^-)$, $1\leq i\leq r$, and $-R_i(n_i^+,n_i^-)$, $r+1\leq i\leq r+s$, each corresponding to a 2-handle, which are split from each other. 
For $1\leq i\leq r$, the sublink corresponding to the $i^{\text{th}}$ 2-handle $h_i$ consists of $n_i^+$ copies of $R$ and $n_i^-$ copies of $r(R)$, all parallel to each other. 
For $r+1\leq i\leq r+s$, the sublink corresponding to the $i^{\text{th}}$ 2-handle consists of $n_i^+$ copies of $-R$ and $n_i^-$ copies of $m(R)=r(-R)$, all parallel.
We will frequently conflate a relative slice (or a relative concordance) with the corresponding planar surface.

Links consisting of $K$ and collections of parallel copies of $R$, $r(R)$, $-R$, and $m(R)$ will play a significant role in our analysis of relative concordance. Given $r,s\geq 0$ and a $2(r+s)$-tuple $n=((n_1^+,n_1^-),\dots,(n_{r+s}^+,n_{r+s}^-))$, we define the link $R(r,n)$ by \[R(r,n):=\left(\sqcup_{i=1}^rR_i(n_i^+,n_i^-)\right)\sqcup\left(\sqcup_{i=r+1}^{r+s}-R_i(n_i^+,n_i^-)\right),\] and we define \[K(R,r,n):=K\sqcup R(r,n).\]

\begin{lem} \label{prop:unknot}
A knot $K$ is slice relative to $R$ if and only if $K$ is relatively concordant to the unknot.
\end{lem}

\begin{proof} Consider a relative concordance $C$ from $K\subset S^3\times\{ 1 \}$ to the unknot in $S^3\times\{ 0 \}$. There may be boundary components of $C$ that are parallel copies of the knots $R_s\subset S^3\times\{ 0 \}$.
These split knots are contained in disjoint balls $B_i\subset S^3\times \{0\}$. Choose arcs, disjoint from each other and from $C$, connecting these balls to disjoint points in $S^3\times \{ 1\}$. Viewing each ball $B_i$ as a neighborhood of a point, the corresponding arc guides an isotopy of $B_i$ to an embedding into $S^3\times \{1\}$. After capping off the unknot with a disk, this gives rise a planar surface $C'$ bounded by $K$ with possible additional boundary components parallel to a split collection of copies of $R$ and $-R$ in $S^3\times \{ 1\}$.

Conversely, starting from a knot $K$ that is slice relative to $R$, use an arc from the slice disk to $S^3\times \{ 0\}$ to produce an unknot, and analogously to the argument in the preceding paragraph, reposition the planar surface so its boundary components parallel to $-R$ in $S^3\times \{ 1\}$ are replaced with those parallel  to $R$ in $S^3\times \{ 0\}$.
\end{proof}

\begin{prop} \label{prop:equiv}
Concordance relative to $R$ is an equivalence relation.
\end{prop}

\begin{proof} 
Proofs of both the symmetry (turning a relative concordance from $J$ to $K$ upside down) and the reflexivity (considering the product cylinder $K\times [0,1]\subset S^3\times [0,1]$) are immediate. 

Now suppose $C_1\subset S^3\times[0,1]$ is a relative concordance from $J$ to $K$, and $C_2\subset S^3\times[1,2]$ is a relative concordance from $K$ to $L$. Some of the boundary components of the planar surface $C_1$ are parallel copies of a split link $R_r\subset S^3\times \{ 1\}$, and also some boundary components of $C_2$ are parallel copies of $R_s\subset S^3\times \{ 1\}$. Applying an isotopy if necessary, we may assume that the balls containing these links $R_r$ and $R_s$ are disjoint. By definition, the components of $R_r$ are contained in a collection of disjoint $3$-balls $\sqcup B_i\subset S^3\times\{ 1\}$, disjoint from $K$. Consider a collection of disjoint arcs in $S^3\times [1,2]\smallsetminus C_2$ connecting $R_r$ to some points in $S^3\times \{ 2\}\smallsetminus \partial C_2$; we use these arcs to isotope $C_1$ so that all its boundary components, except for $J$ and $K$, are parallel copies of a disjoint union of $r$ copies of $R$ in $S^3\times \{ 2\}\smallsetminus \partial C_2$. An analogous isotopy of $C_2$ moves its boundary components in $R_s$ to $S^3\times \{ 0\}$. The union of these isotopic versions of $C_1$ and $C_2$ provides a relative concordance from $J$ to $L$, showing that the relation is transitive.
\end{proof}

The steps of the proof of the following proposition -- that the connected sum operation is well defined on equivalence classes, the identity elements and the inverse -- are directly analogous to the proof in the classical knot concordance context.

\begin{prop} \label{prop:group}
The set of concordance classes of knots relative to $R$ along with the operation of connected sum forms a group.
\end{prop}

As the results in this paper concern concordance in the locally flat category, we now restrict to that setting.
Let $\mathcal{C}$ denote the topological concordance group, and let $\mathcal{C}_R$ denote the group of topological concordance classes of knots relative to $R$.
We call $\mathcal{C}_R$ the \emph{knot concordance group relative to $R$} or \emph{relative concordance group}.
Proposition~\ref{prop:group}, along with the observation that any slice knot is relatively slice, implies there is a surjective homomorphism \[\text{rel}_R:\mathcal{C}\twoheadrightarrow\mathcal{C}_R\] with kernel $\mathcal{R}\leq\mathcal{C}$, the subgroup consisting of concordance classes of knots which are slice relative to $R$.


\section{Seifert forms and algebraic relative concordance} \label{sec:Seifert}

Let $L\subset S^3$ be an oriented link.
A \emph{Seifert surface} for $L$ is a compact oriented surface $F$ with $\partial F=L$. 
We require that $F$ has no closed components. 
Since $F$ is oriented, it admits a regular neighborhood homeomorphic to $F\times[-1,1]$, where we identify $F$ with $F\times\{0\}$. 
In particular, we have the positive and negative push-offs of $F$, identified with $F\times\{1\}$ and $F\times\{-1\}$, respectively. 
The \emph{Seifert form} for $L$ with respect to $F$ is the bilinear pairing
\[B_F:H_1(F)\times H_1(F)\to \mathbb{Z} \] defined by $B_F(x,y)=\text{lk}(x,y^+)$, the linking number of $x$ with the positive push-off of $y$. 
If we choose a basis for $H_1(F)$, we may write this pairing as a \emph{Seifert matrix} $A_F$. 
In this section, we analyze the relationship between Seifert forms and relative concordance. 

Fix a Seifert surface $F_R$ for $R$.
Then $r(R)$, $-R$, and $m(R)$ have Seifert surfaces $-F_R$, $-m(F_R)$, and $m(F_R)$, respectively.
The boundary link $K(R,r,n)$ bounds a Seifert surface $F_K(R,r,n)$ with $(1+\sum_{i=1}^{r+s}(n_i^++n_i^-))$ components consisting of a Seifert surface $F_K$ for $K$ and a surface $F_R(r,n)$ consisting of copies of $F_R$, $-F_R$, $-m(F_R)$, and $m(F_R)$ corresponding to the copies of $R$, $r(R)$, $-R$, and $m(R)$, respectively, in $R(r,n)$.

Now recall the following classical knot concordance result (see \cite[Proposition 8.17]{Lickorish}).

\begin{prop} \label{prop:metabolizer}
If $K\subset S^3$ is slice, then for any Seifert surface $F$ for $K$ there is a half-dimensional summand of $H_1(F)$ on which the corresponding Seifert form vanishes.
In other words, any Seifert form for $K$ is metabolic.
\end{prop}

We prove an analogous statement in the context of relative concordance.

\begin{prop} \label{prop:halfdies}
If $K\subset S^3$ is slice relative to $R$, then for some $r,s\geq 0$ there exists a $2(r+s)$-tuple of nonnegative integers $n$ such that for any Seifert surfaces $F_K$ for $K$ and $F_R$ for $R$, there is a half-dimensional summand of $H_1(F_K(R,r,n))$ on which the corresponding Seifert form vanishes. 
That is, there exist $r,s\geq 0$ and $n$ such that the Seifert form for $K(R,r,n)$ with respect to $F_K(R,r,n)$ is metabolic.
\end{prop}

\begin{proof}
Just as in the classical setting, this is a consequence of the ``half lives, half dies" principle. 
Suppose $K$ is slice relative to $R$, so that $K(R,r,n)$ bounds some connected genus 0 surface $C$ in $B^4$, for some $r,s\geq 0$ and $n$. 
Choose any Seifert surfaces $F_K$ for $K$ and $F_R$ for $R$, and let $F$ be the closed surface $F=C\cup_{K(R,r,n)}F_K(R,r,n)$, where $F_K(R,r,n)$ is the surface described above. 

By a Mayer-Vietoris argument using the decomposition $B^4=E_C\cup_{\partial\nu C}\nu C$, where $E_C$ is the exterior of $C$, we see that $H_1(E_C)\cong\Z$, generated by the meridian of the surface.
Define a map $E_{K(R,r,n)}\to S^1\subset\C$ which maps the regular neighborhood $F_K(R,r,n)\times[-1,1]$ of $F_K(R,r,n)$ by projection to $[-1,1]$ followed by the map $t\mapsto e^{t\pi i}$, and which sends the rest of $E_{K(R,r,n)}$ to $-1\in S^1$. 
This maps the meridian of each component of $K(R,r,n)$ to $S^1$ by a degree 1 map.
Extend this map to $\partial\nu C\cong C\times S^1$ by a map which on $\partial\nu C$ is isotopic to projection to the second factor.
This extends the map to all of $\partial E_C$.
Since $S^1$ is a model for the classifying space $K(\Z,1)$ and $H_1(E_C)\cong\Z$, there is no obstruction to extending this map to $E_C$.
The transverse preimage of $1\in S^1$ yields a 3-manifold $M\subset B^4$ with $\partial M=F$. 

We now sketch the remainder of the proof, which is identical to the classical setting, and refer the reader to \cite[Chapter 8]{Lickorish} for more details.
Let $g$ be the genus of $F$, which is equal to the genus of $F_K(R,r,n)$.
The inclusion-induced homomorphism $H_1(F;\Q)\to H_1(M;\Q)$ has kernel $K_\Q\cong\Q^g$, and there is a basis $\{\alpha_1,\dots,\alpha_g,\beta_1,\dots,\beta_g\}$ for $H_1(F;\Z)$ such that the images of the $\alpha_i$ in $H_1(F;\Q)$ form a basis for $K_\Q$.
In particular, $\alpha_i\mapsto 0$ under $H_1(F;\Z)\to H_1(M;\Q)$ for all $1\leq i\leq g$.
The Seifert form for $F_K(R,r,n)$ vanishes on $\text{span}(\{\alpha_1,\dots,\alpha_g\})$, which is a half-dimensional summand of $H_1(F;\Z)=H_1(F_K(R,r,n);\Z)$.
\end{proof}

This result motivates the definitions of \emph{algebraically relatively slice} knots and the \emph{algebraic relative concordance groups}. 
Just as with the classical algebraic concordance group, the algebraic relative concordance groups will be defined using Witt equivalence of Seifert forms.

\begin{defn} \label{def:Seifert}
A \emph{knot Seifert form} is a bilinear form $B$ on $\Z^{2g}$ for some $g$ satisfying $\det(B-B^T)=\pm 1$.
A \emph{knot Seifert matrix} is a matrix corresponding to some knot Seifert form.
\end{defn}

Now fix a knot type $R\subset S^3$ and a Seifert surface $F_R$ for $R$.
Denote the corresponding (knot) Seifert form by $B_R$. 
Given $r,s\geq 0$ and $n=((n_1^+,n_1^-),\dots,(n_{r+s}^+,n_{r+s}^-))$, denote the Seifert form for the link $R(r,n)$ with respect to the surface $F_R(r,n)$ by $B_R(r,n)$.

\begin{defn} \label{def:algrelslice}
A knot Seifert form $B$ is \emph{algebraically slice relative to $R$}, or \emph{algebraically relatively slice}, if there exist $r,s\geq 0$ and $n=((n_1^+,n_1^-),\dots,(n_{r+s}^+,n_{r+s}^-))$ such that the form $B\oplus B_R(r,n)$ is metabolic, that is, there is a half-dimensional summand on which the form $B\oplus B_R(r,n)$ vanishes.
\end{defn}

Note that algebraically slice knots are algebraically slice relative to any $R$ with $r,s=0$.

\begin{prop} \label{prop:welldef}
The above definition is independent of the choice of Seifert surface $F_R$.
\end{prop}

\begin{proof}
We show that any two Seifert surfaces $F_R$ and $F_R'$ for $R$ yield Seifert forms $B_R$ and $B_R'$ such that $B\oplus B_R(r,n)$ and $B\oplus B_R'(r,n)$ are \emph{s-equivalent}.
Recall that s-equivalence is the relation on Seifert forms which accounts for (de)stabilizations of a Seifert surface.
The existence of a metabolizer is preserved by s-equivalence \cite{LivingstonNaik}.

Let $F_R$ and $F_R'$ be Seifert surfaces for $R$.
Recall that $F_R$ and $F_R'$ differ by a sequence of (de)stabilizations (hence the corresponding Seifert forms $B_R$ and $B_R'$ are s-equivalent).
We recall a standard proof of this fact, which is outlined in \cite[Section 9]{Gordon}, to see that $B\oplus B_R(r,n)$ and $B\oplus B_R'(r,n)$ are s-equivalent.

Similar to our proof of Proposition~\ref{prop:metabolizer}, we may begin with degree 1 maps $f$ and $g$ from the knot exterior $E_R$ to $S^1$ with transverse preimages $f^{-1}(1)=F_R$ and $g^{-1}(1)=F_R'$.
These maps extend to a homotopy rel boundary $H:E_R\times[0,1]\to S^1$.
The transverse preimage $H^{-1}(1)$ is a 3-manifold $M\subset E_R\times[0,1]$ with boundary $F_R\cup(R\times[0,1])\cup F_R'$.
After discarding closed components, a handle deomposition of this 3-manifold relative to $F_R$ yields the desired sequence of (de)stabilizations.
The same argument applies to $-m(F_R)$ and $-m(F_R')$.

The transverse preimage $H^{-1}(1)$ also pulls back the normal bundle of $1\in S^1$ to the normal bundle of $M$, so parallel copies of $F_R$ and $F'_R$ (or their orientation reverses) cobound parallel copies of $M$ (or its orientation reverse) in $E_R\times[0,1]$.
Again, the same argument applies to parallel copies of $-m(F_R)$ and $-m(F_R')$ (and their orientation reverses).
In particular, this is true for split collections of appropriate numbers of parallel copies of $F_R$ and $F'_R$ as well as $-m(F_R)$ and $-m(F'_R)$ (and their orientation reverses) corresponding to the link $R(r,n)$.
The same handle decompositon relative to $F_R$ on each copy of $M$ (and corresponding decompositions on copies of $-M$, $-m(M)$, and $m(M)$) yields a sequence of (de)stabilizations which correspond to an s-equivalence between the forms $B_R(r,n)$ and $B_R'(r,n)$.
As $B_R(r,n)$ and $B_R'(r,n)$ are s-equivalent, so are $B\oplus B_R(r,n)$ and $B\oplus B_R'(r,n)$.
\end{proof}

As in the setting of classical algebraic concordance, a relative version of Definition~\ref{def:algrelslice} yields an equivalence relation on knot Seifert forms.

\begin{defn} \label{def:algrelconc}
Two knot Seifert forms $B_1$ and $B_2$ are \emph{algebraically concordant relative to $R$}, or \emph{algebraically relatively concordant}, if the form $B_1\oplus -B_2$ is algebraically slice relative to $R$.
\end{defn}

Before showing this is an equivalence relation, recall that given a Seifert form $B_R$ for $R$, the forms $B_R^T$, $-B_R$, and $-B_R^T$ are Seifert forms for $r(R)$, $m(R)$, and $-R$, respectively; see the proof of \cite[Proposition 6.12]{Lickorish}. 
Further recall that the forms $B_R$ and $B_R^T$ are algebraically concordant, that is, $B_R\oplus -B_R^T$ is metabolic.

\begin{prop} \label{prop:algequiv}
Algebraic concordance relative to $R$ is an equivalence relation on knot Seifert forms.
\end{prop}

\begin{proof}
This follows from a similar argument to the classical case. 
Given any knot Seifert form $B$, the form $B\oplus -B$ is algebraically slice, and is therefore algebraically relatively slice, so $B\sim B$.

If $B_1\oplus -B_2$ is algebraically relatively slice, where $B_1\oplus -B_2\oplus B_R(r,n)$ is metabolic, then a change of basis shows $-B_2\oplus B_1\oplus B_R(r,n)$ is metabolic, hence \[B_2\oplus -B_1\oplus B_R(s,n')\cong B_2\oplus -B_1\oplus -B_R(r,n)\] is metabolic, where \[n'=((n_{r+1}^-,n_{r+1}^+),\dots,(n_{r+s}^-,n_{r+s}^+),(n_1^-,n_1^+),\dots(n_r^-,n_r^+)),\] and where the isomorphism again follows from a change of basis.
Thus, $B_1\sim B_2$ implies $B_2\sim B_1$. 

Transitivity follows from the Witt cancellation lemma; see \cite{LivingstonNaik}. 
If $B_1 \oplus -B_2 \oplus B_R(r_1,n_1)$ and $B_2\oplus -B_3 \oplus B_R(r_2,n_2)$ are metabolic, then so is their direct sum, which after a change of basis has the form \[B_1\oplus-B_3\oplus B_R(r_1,n_1)\oplus B_R(r_2,n_2)\oplus B_2\oplus -B_2.\] 
By a change of basis, $B_R(r_1,n_1)\oplus B_R(r_2,n_2)\cong B_R(r_1+r_2,n)$, where $n$ is the tuple beginning with the first $r_1$ pairs of $n_1$, followed by the first $r_2$ pairs of $n_2$, and then by the last $s_1$ and $s_2$ pairs of $n_1$ and $n_2$, respectively.
Thus, since $B_2\oplus -B_2$ is metabolic, the Witt cancellation lemma implies  $B_1\oplus -B_3\oplus B_R(r_1+r_2,n)$ is metabolic.
Thus, $B_1\sim B_2$ and $B_2\sim B_3$ implies $B_1\sim B_3$.
\end{proof}

Just as in the classical setting, we obtain a group structure on the set of algebraic concordance classes relative to $R$ of knot Seifert forms.

\begin{defn} \label{def:alggroup}
The \emph{algebraic concordance group relative to $R$}, or \emph{algebraic relative concordance group}, is the set of algebraic concordance classes relative to $R$ of knot Seifert forms with the operation of direct sum, denoted $\mathcal{AC}_R$.
\end{defn}

As algebraically slice knots are algebraically relatively slice, the above definition yields a surjective homomorphism \[\mathcal{A}\text{rel}_R:\mathcal{AC}\twoheadrightarrow\mathcal{AC}_R\] sending an algebraic concordance class to the corresponding algebraic relative concordance class.
In the language of Definition~\ref{def:algrelslice}, Proposition~\ref{prop:halfdies} states that relatively slice knots are algebraically relatively slice, hence there is a homomorphism \[A_R:\mathcal{C}_R\to\mathcal{AC}_R.\]
It fits into the commutative square below and is therefore surjective.

\begin{equation*}
\begin{tikzcd}
\mathcal{C} \arrow[r,twoheadrightarrow, "{\text{rel}}_R"] \arrow[d,twoheadrightarrow,"A"']& \mathcal{C}_R \arrow[d,"A_R"] \\
\mathcal{AC} \arrow[r,twoheadrightarrow, "\mathcal{A}{\text{rel}}_R"'] & \mathcal{AC}_R
\end{tikzcd}
\end{equation*}

As we see in Section~\ref{sec:cables}, the form $B_R(r,n)$ is also the Seifert form of a knot which is the connected sum of cables of $R$ and $-R$, the number and type of which are dictated by $n$.
Thus, a knot Seifert form $B$ is algebraically slice relative to $R$ if and only if there exist $r,s\geq 0$ and $n$ such that $B$ and $-B_R(r,n)$ are algebraically concordant.
It will sometimes be convenient to speak of the algebraic concordance class of $B_R(r,n)$ (as a knot Seifert form), despite the fact that algebraic concordance of links is a much more subtle notion; recent progress on this includes \cite{CCS,Simian}.


\section{Algebraic relative concordance and cables} \label{sec:cables}

In this section, we explicitly describe the relationship between the algebraic concordance group $\mathcal{AC}$ and the algebraic relative concordance group $\mathcal{AC}_R$ for any fixed $R$, as summarized in Proposition~\ref{prop:quotient} below.
For a knot $R\subset S^3$, recall that $\mathcal{R}\leq\mathcal{C}$ is the subgroup consisting of knots which are slice relative to $R$.
For $n\in\Z$, let $K_{n,1}$ be the $(n,1)$-cable of $K\subset S^3$.
Recall that these cables are obtained by performing specific satellite operations.
For our study of the group $\mathcal{AC}_R$, the relevant cables are $R_{n,1}$ and $(-R)_{n,1}$, $n\in\Z$.

\begin{prop} \label{prop:quotient}
For any fixed knot $R\subset S^3$, the homomorphism $\mathcal{A}\text{rel}_R:\mathcal{AC}\twoheadrightarrow\mathcal{AC}_R$ has kernel $A(\mathcal{R})=\langle[R_{n,1}],[(-R)_{n,1}]\rangle_{n\in\Z}$.
Thus, we have isomorphisms \[\mathcal{AC}_R\cong\mathcal{AC}/A(\mathcal{R})=\mathcal{AC}/\langle[R_{n,1}],[(-R)_{n,1}]\rangle_{n\in\Z}.\]
\end{prop}
By Proposition~\ref{prop:quotient}, we may enlarge the commutative square from the end of Section~\ref{sec:Seifert} to diagram \eqref{diagram:exact}, copied below.
\begin{equation*}
\begin{tikzcd}
0 \arrow[r] & \mathcal{R} \arrow[r,hookrightarrow] \arrow[d,twoheadrightarrow,"A|_\mathcal{R}"'] & \mathcal{C} \arrow[r,twoheadrightarrow,"{\text{rel}}_R"] \arrow[d,twoheadrightarrow,"A"']& \mathcal{C}_R \arrow[d,twoheadrightarrow,"A_R"] \arrow[r] & 0 \\
0 \arrow[r] & A(\mathcal{R}) \arrow[r,hookrightarrow] & \mathcal{AC} \arrow[r,twoheadrightarrow,"\mathcal{A}{\text{rel}}_R"'] & \mathcal{AC}_R \arrow[r] & 0
\end{tikzcd} \tag{$\ast$}
\end{equation*}

\begin{proof}
We first establish the isomorphism $\mathcal{AC}_R\cong\mathcal{AC}/\langle[R_{n,1}],[(-R)_{n,1}]\rangle_{n\in\Z}$, which identifies $\langle[R_{n,1}],[(-R)_{n,1}]\rangle_{n\in\Z}$ as the kernel of $\mathcal{A}\text{rel}_R$.
Then, we show $A(\mathcal{R})=\langle[R_{n,1}],[(-R)_{n,1}]\rangle_{n\in\Z}$.

Fix a Seifert surface $F_R$ for $R$.
Given $r,s$ and $n$, consider the link $R(r,n)$, its Seifert surface $F_R(r,n)$, and the associated Seifert form $B_R(r,n)$.
Banding together each collection of parallel copies of $F_R$ and $-F_R$ (resp. $-m(F_R)$ and $m(F_R)$) in $F_R(r,n)$ in a prescribed way yields a surface with Seifert form $B_R(r,n)$ and boundary $L=\sqcup_{i=1}^{r+s} L_i$, where $L_i=R_{(n_i^+-n_i^-),1}$ for $1\leq i\leq r$ and $L_i=(-R)_{(n_i^+-n_i^-),1}$ for $r+1\leq i\leq r+s$.
Thus, the form $B_R(r,n)$ is also a Seifert form for the boundary link $L$.
It is the direct sum of the Seifert forms of the $L_i$ and therefore represents an element of the subgroup $\langle[R_{n,1}],[(-R)_{n,1}]\rangle_{n\in\Z}\leq\mathcal{AC}$.

Now suppose $B_1$ and $B_2$ are algebraically relatively concordant Seifert forms.
This is equivalent to the existence of a $2(r+s)$-tuple $n$ such that $B_1\oplus-B_2\oplus B_R(r,n)$ is metabolic.
Equivalently, $B_1$ is algebraically concordant to $B_2\oplus-B_R(r,n)$.
Thus, the images of $B_1$ and $B_2$ in the quotient $\mathcal{AC}/\langle[R_{n,1}],[(-R)_{n,1}]\rangle_{n\in\Z}$ agree as $-B_R(r,n)\cong B_R(s,n')\in\langle[R_{n,1}],[(-R)_{n,1}]\rangle_{n\in\Z}$, where $n'$ is as in the proof of Proposition~\ref{prop:algequiv}.
Conversely, if the images of $B_1$ and $B_2$ agree in this quotient, then there exist $r,s$ and $n_1,\dots,n_{r+s}\in \Z$ so that $B_1\oplus -B_2\oplus B_R(r,n)$ is metabolic, where \[n=((\max\{-n_1,0\},\max\{n_1,0\}),\dots(\max\{-n_{r+s},0\},\max\{n_{r+s},0\})).\]
(Specifically, for $1\leq i\leq s$ the integers $n_i$ represent cables of $-R$, and for $s+1\leq i\leq r+s$ they represent cables of $R$.)
Thus, $\ker(\mathcal{A}\text{rel}_R)=\langle[R_{n,1}],[(-R)_{n,1}]\rangle_{n\in\Z}$, and $\mathcal{AC}_R\cong\mathcal{AC}/\langle[R_{n,1}],[(-R)_{n,1}]\rangle_{n\in\Z}$.

Finally, if $K$ is slice relative to $R$, then by Proposition~\ref{prop:halfdies} it is algebraically slice relative to $R$, so $A(\mathcal{R})\leq\ker(\mathcal{A}\text{rel}_R)=\langle[R_{n,1}],[(-R)_{n,1}]\rangle_{n\in\Z}$.
To see the reverse inclusion, we show that all $R_{n,1}$ and $(-R)_{n,1}$ are relatively slice.
Reversing some of the band moves described above, we obtain a genus 0 cobordism $C'$ in $S^3\times[0,1]$ from $R_{n,1}\subset S^3\times\{0\}$ to $|n|$ parallel copies of $R$ (or $r(R)$) in $S^3\times\{1\}$.
We then isotope $C'$ guided by an arc from an interior point of $C'$ to a point $p\in (S^3\times\{1\})\smallsetminus B$, where $B$ is an open ball containing the parallel copies of $R$, otherwise missing $C'$. 
We remove an open disk from $C'$ which is a neighborhood of $p$ in $C'$ after the isotopy.
Call the resulting genus 0 surface $C$.
The $|n|$ parallel copies of $R$ (or $r(R)$) bound parallel disks in a single 0-framed 2-handle $h$ attached to $S^3\times\{1\}$ along $R$.
Capping off $C$ with these disks yields a relative concordance from $R_{n,1}$ to the unknot.
Similarly, we obtain a genus 0 cobordism in $S^3\times[0,1]$ from $(-R)_{n,1}\subset S^3\times\{1\}$ to $|n|$ parallel copies of $R$ (or $r(R)$) in $S^3\times\{1\}$.
After an analogous isotopy and removing an open disk in $S^3\times\{0\}$, we cap off this cobordism with $|n|$ parallel disks in a single 0-framed 2-handle $h$ attached to $S^3\times\{1\}$ along $R$ to obtain a relative concordance from $(-R)_{n,1}$ to the unknot.
Thus, $\langle R_{n,1},(-R)_{n,1} \rangle_{n\in\Z}\leq\mathcal{R}$, hence $\langle[R_{n,1}],[(-R)_{n,1}]\rangle_{n\in\Z}\leq A(\mathcal{R})$.
\end{proof}


\section{Signatures} \label{sec:signatures}

In this section, we examine the signatures of knots which are algebraically slice relative to $R$ and prove the following theorem.
As all signatures we consider are valued in $2\Z$, we will rescale them using the normalization $2\Z\xrightarrow{\cong}\Z$, $2\mapsto 1$.

\summand

Recall that for every symmetric, irreducible polynomial $\lambda_\theta(t)=t^2-2\cos(\theta)t+1\in\R[t,t^{-1}]$, $0<\theta<\pi$, there is an associated homomorphism $\sigma_\theta:\mathcal{AC}\to\Z$ which is surjective for every $\theta$ such that $e^{\theta i}$ is a root of the Alexander polynomial of some knot.
Milnor showed the homomorphisms $\{\sigma_\theta\}_{\theta\in S}$, where $S=\{0<\theta<\pi\,|\,\cos(\theta)=\frac{2k-1}{2k}, k\in\N \}$, yield a surjection $\mathcal{AC}\twoheadrightarrow\Z^\infty$ \cite{Milnor}. 
Defining $\theta_k\in S$ such that $\cos(\theta_k)=\frac{2k-1}{2k}$, Milnor exhibited knots $K_j$ such that $\sigma_{\theta_i}(K_j)=\delta_{ij}$.
Levine later showed that the collection of all such signatures $\sigma_\theta$ is a complete set of invariants for the torsion-free part of $\mathcal{AC}$ \cite{Levine}.
To prove Theorem~\ref{thm:main}, we will leverage the isomorphism $\mathcal{AC}_R\cong\mathcal{AC}/\langle [R_{n,1}],[(-R)_{n,1}]\rangle_{n\in\Z}$ from Proposition~\ref{prop:quotient} and study Milnor's surjection $\mathcal{AC}\twoheadrightarrow\Z^\infty$.
For the convenience of the reader, we recall two equivalent definitions of the signatures $\sigma_\theta$.
For a detailed survey on various definitions of these signatures and how they are related, see \cite{Conway}.


\subsection{Milnor signatures} \label{subsec:Milnorsgn}

First, we recall Milnor's definition \cite{Milnor}. 
For a knot $K\subset S^3$, let $E_K^\infty$ denote the infinite cyclic cover of the knot exterior $E_K$, the cover corresponding to the abelianization homomorphism $\pi_1(E_K)\twoheadrightarrow H_1(E_K;\Z)\cong\Z$. 
Milnor showed that the cup product \[\cup:H^1( E_K^\infty;\R)\times H^1(E_K^\infty,\partial E_K^\infty;\R)\to H^2(E_K^\infty,\partial E_K^\infty;\R)\cong\R\] defines a nonsingular, skew-symmetric $\R$-bilinear form. 
Using this pairing and the isomorphism $H^1(E_K^\infty ,\partial E_K^\infty;\R)\xrightarrow{\cong}H^1( E_K^\infty;\R)$, Milnor defined a symmetric bilinear pairing called the \emph{quadratic form} of $K$ \[H^1(E_K^\infty;\R)\times H^1( E_K^\infty;\R)\to\R,\] given by $(x,y)\mapsto (t^*x)\cup y+(t^*y)\cup x$, where $t$ is the deck transformation corresponding to $1\in\Z\cong H_1(E_K;\Z)$. 
Since $\R[t,t^{-1}]$ is a PID and $H^1( E_K^\infty ;\R)$ is a finitely-generated torsion $\R[t,t^{-1}]$-module, $H^1(E_K^\infty ;\R)$ decomposes as \[H^1(E_K^\infty;\R)\cong\left(\bigoplus_{\lambda=\overline{\lambda}} H^1_\lambda( E_K^\infty;\R)\right)\oplus\left(\bigoplus_{\lambda\neq\overline{\lambda}} \big(H^1_\lambda(E_K^\infty;\R)\oplus H^1_{\overline{\lambda}}(E_K^\infty;\R)\big)\right),\] where the direct sum runs over all monic irreducible polynomials $\lambda(t)\in\R[t,t^{-1}]$, and where $H^1_\lambda(E_K^\infty ;\R)$ is the $\lambda(t)$-primary part of $H^1(E_K^\infty;\R)$.
Note that the direct sum is split between symmetric and non-symmetric polynomials.

The only contributions to the signature of the quadratic form of $K$ come from symmetric polynomials $\lambda(t)=\lambda_\theta(t)=t^2-2\cos(\theta)t+1$, $0<\theta<\pi$, whose roots are $e^{\pm\theta i}$. 
Milnor defined the signature $\sigma_\theta(K)$ to be the signature of the quadratic form of $K$ restricted to the $\lambda_\theta$-primary part of $H^1(E_K^\infty;\R)$.
He then proved that $\sigma_\theta:\mathcal{C}\twoheadrightarrow\Z$ is a homomorphism. 
In fact, $\sigma_\theta$ decends to a homomorphism on the algebraic concordance group, since it only depends on the cohomology of the infinite cyclic cover of $E_K$.
Milnor showed that the homomorphism $\oplus_{\theta\in S}\sigma_\theta:\mathcal{AC}\twoheadrightarrow\Z^\infty$ is a surjection.


\subsection{Levine signatures} \label{subsec:Levinesgn}

We now recall Levine's definition \cite{Levine}.
Let $\F$ be a field with $\text{char}(\F)\neq 2$.
An \emph{isometric structure} over $\F$ consists of a triple $(V,B,T)$, where $V$ is an even-dimensional $\F$-vector space, $B$ is a nonsingular, symmetric bilinear form on $V$, and $T$ is an isometry of $V$. 
An isometric structure $(V,B,T)$ is said to be \emph{null-cobordant} if $B$ has a $T$-invariant metabolizer, and two isometric structures $(V_1,B_1,T_1)$ and $(V_2,B_2,T_2)$ are \emph{cobordant} if the bilinear form $B_1\oplus -B_2$ on $V_1\oplus V_2$ has a metabolizer invariant under the isometry $T_1\oplus T_2$.
Let $\mathcal{G}_\F$ denote the group of cobordism classes of isometric structures $(V,B,T)$ over $\F$ satisfying $\Delta_T(1)\Delta_T(-1)\neq 0$, where $\Delta_T$ is the characteristic polynomial of $T$.
Using the decomposition of any $V$ into generalized eigenspaces, there is a decomposition $\mathcal{G}_\F\cong\bigoplus_{\lambda}\mathcal{G}^\lambda_\F$, where $\lambda=\lambda(t)$ is a symmetric, monic, irreducible polynomial over $\F$, and $\mathcal{G}^\lambda_\F$ is the subgroup of isometric structures $(V,B,T)$ with $\Delta_T$ some power of $\lambda$.

Levine defines an injective homomorphism $\mathcal{AC}\hookrightarrow\mathcal{G}_\Q$ which represents the class of a Seifert matrix by a $2g\times 2g$ nonsingular matrix $A$ and sends it to the isometric structure $(\Q^{2g},A+A^T,A^{-1}A^T)$.
For an irreducible, symmetric polynomial $\lambda$ over $\R$, Levine defines $\sigma_\lambda:\mathcal{AC}\to\Z$ as the composite \[\mathcal{AC}\hookrightarrow\mathcal{G}_\Q\to\mathcal{G}_\R\twoheadrightarrow\mathcal{G}^\lambda_\R\twoheadrightarrow\Z.\]
The final homomorphism takes the signature of the bilinear form $(V,B)$ associated to an isometric structure $(V,B,T)$, restricted to the $\lambda$-primary component. 
For more details, see \cite{Collins,LivingstonNaik}. 
Since $\lambda$ is a symmetric, irreducible polynomial in $\R[t,t^{-1}]$, we have $\lambda(t)=\lambda_\theta(t)=t^2-2\cos(\theta)t+1$ for some $0<\theta<\pi$. 

It is well-known that $\sigma_{\lambda_\theta}=\sigma_\theta$, that is, that the Levine and Milnor signatures agree \cite{Levine}.
Henceforth, we will use the terms ``Levine signatures'' and ``Milnor signatures'' interchangeably.
As $\cos(-\theta)=\cos(\theta)$, we will use the notation $\lambda_\theta$ and $\sigma_\theta$ to mean $\lambda_{-\theta}$ and $\sigma_{-\theta}$, respectively, when convenient.


\subsection{Proof of Theorem~\ref{thm:main}} \label{subsec:proof}

We now leverage Proposition~\ref{prop:quotient} to establish a surjection $\mathcal{AC}_R\twoheadrightarrow\Z^\infty$ for any fixed $R$.

\begin{proof}[Proof of Theorem~\ref{thm:main}]
Fix a knot $R\subset S^3$.
We first reduce the proof to a number-theoretic argument.
We will use a commutative square of surjections
\begin{center}
\begin{tikzcd}
\mathcal{AC} \arrow[r,twoheadrightarrow,"\oplus_{S'}\sigma_\theta"] \arrow[d,twoheadrightarrow] & \Z^\infty \arrow[d,twoheadrightarrow] \\
\mathcal{AC}_R \arrow[r,twoheadrightarrow] & \Z^\infty/\langle(\oplus_{S'}\sigma_\theta)([R_{n,1}]),(\oplus_{S'}\sigma_\theta)([(-R)_{n,1}])\rangle_{n\in\Z}
\end{tikzcd} 
\end{center}
\noindent Recall that $\mathcal{AC}_R\cong\mathcal{AC}/\langle[R_{n,1}],[(-R)_{n,1}]\rangle_{n\in\Z}$ by Proposition~\ref{prop:quotient}, so the square commutes.
The homomorphism $\oplus_{S'}\sigma_\theta$ is related to Milnor's surjection $\oplus_S\sigma_\theta:\mathcal{AC}\twoheadrightarrow\Z^\infty$. 
Define \[S'=\left\{0<\theta<\pi\,\bigg|\,\cos(\theta)=\frac{2p-1}{2p},\,p\,\text{ prime} \right\}\subset S,\]
and
\[S_{\cos}:=\cos(S')=\left\{
\frac{2p-1}{2p}:p\text{ prime}
\right\}.\]
Since $\cos$ is injective on $(0,\pi)$, passage from $S'$ to
$S_{\cos}$ loses no information.

Now recall a result of Seifert \cite{Seifert} which states that the Alexander polynomial $\Delta_{P(K)}(t)$ of a satellite knot $P(K)$ with pattern $P\subset S^1\times D^2$ and companion knot $K\subset S^3$ factors as $\Delta_{P(K)}(t)=\Delta_K(t^n)\cdot\Delta_{P(U)}(t)$, where $[P]=n\in\Z= H_1(S^1\times D^2)$ and $U$ is the unknot.
For $(n,1)$-cables, this becomes $\Delta_{K_{n,1}}(t)=\Delta_K(t^n)$, so that the roots of $\Delta_{K_{n,1}}(t)$ are precisely the $n$-th roots of the roots of $\Delta_K(t)$.
This implies that the Milnor signatures $\sigma_\theta$ of $R_{n,1}$ vanish for all $0<\theta<\pi$, except possibly at those $\theta$ such that $\lambda_{n\theta}(t)$ is a factor of $\Delta_R(t)$, and likewise for the Milnor signatures of $(-R)_{n,1}$.
Recall that $\Delta_{-R}(t)=\Delta_R(t)$, so the Milnor signatures of $(-R)_{n,1}$ do not contribute any additional values of $\theta$ we must consider.

Let $\Theta=\{0<\beta<\pi\,|\,\Delta_R(e^{\beta i})=0\}$.
As $\Theta$ is a finite set, the result will follow if we show that, for each $\beta\in\Theta$, the set \[S_\beta=\left\{\cos\left(\frac{\beta+2\pi m}{n}\right):m\in\mathbb Z,\ n\ge 1\right\}\] has finite intersection with the set $S_{\text{cos}}$.
The desired homomorphism $\mathcal{AC}_R\twoheadrightarrow\Z^\infty$ is then given by projecting $\Z^\infty/\langle(\oplus_{S'}\sigma_\theta)([R_{n,1}]),(\oplus_{S'}\sigma_\theta)([(-R)_{n,1}])\rangle_{n\in\Z}$ onto all factors of $\Z^\infty$ for which no $\sigma_\theta(R_{n,1})$, $\theta\in S'$, $n\in\Z$, is nonzero.
The following number-theoretic argument concludes the proof by showing that the intersection $S_{\text{cos}}\cap S_\beta$ is finite for every $\beta\in\R$. 

\begin{prop}\label{prop: finite}
For every real $\beta$, the set
\[ S_{\cos}\cap S_\beta \]
has at most one element.
\end{prop}

\begin{proof}[Proof of Proposition~\ref{prop: finite}]
For $p$ a prime, let
\[
q_p=\frac{2p-1}{2p},
\]
and suppose $q_p\in S_\beta$. Then for some $m\in\mathbb Z$ and $n\ge 1$,
\[
q_p=\cos\left(\frac{\beta+2\pi m}{n}\right).
\]
By the Chebyshev identity $T_n(\cos\varphi)=\cos(n\varphi)$,
\[
T_n(q_p)=\cos\beta.
\]
Now set
\[
y_p=2q_p=\frac{2p-1}{p},
\qquad
P_n(Y)=2T_n(Y/2).
\]
These rescaled Chebyshev polynomials $P_n$ (sometimes called the
Vieta-Lucas polynomials) satisfy 
\[
P_0(Y)=2,\qquad P_1(Y)=Y,\qquad P_{r+1}(Y)=Y P_r(Y)-P_{r-1}(Y),
\]
so $P_n\in\mathbb Z[Y]$ is monic of degree $n$ for every $n\ge 1$. Also
\[
P_n(y_p)=2T_n(q_p)=2\cos\beta. \tag{1}
\]

Define $y_p=a/p$, where $a=2p-1$. Since $p\nmid a$, this is in lowest terms. Since $P_n$ is monic,
\[
P_n(Y)=Y^n+b_{n-1}Y^{n-1}+\cdots+b_0,
\qquad b_i\in\mathbb Z.
\]
Therefore
\[
P_n(a/p)=
\frac{a^n+b_{n-1}a^{n-1}p+\cdots+b_0p^n}{p^n}.
\]
The numerator is congruent to $a^n$ modulo $p$, hence is not divisible by $p$. Thus $P_n(y_p)$, in lowest terms, has denominator exactly $p^n$.

Consequently, if $2\cos\beta$ is irrational, (1) is impossible. If
\[
2\cos\beta=\frac{A}{B}
\]
in lowest terms with $B>0$, then (1) implies $B=p^n$. A fixed positive integer $B$ is a positive power of at most one prime. Hence at most one prime $p$ can occur, and therefore $|S_{\cos}\cap S_\beta|\le 1$. 
\end{proof}

This concludes the proof of Theorem~\ref{thm:main}.
\end{proof}


\section{Torsion in relative concordance groups} \label{sec:torsion}

Recall that a knot $K\subset S^3$ is \emph{negative amphichiral} if $K$ is isotopic to $-K$.
By definition, any negative amphichiral knot which is not (topologically or smoothly) slice has order 2 in the (topological or smooth) concordance group.
Thus, one readily obtains examples of knots which have order 2 in $\mathcal{C}$.
In contrast, the following question remains open.

\begin{quest}[K3 Problem 1.38(a) \cite{K3}] \label{q:torsion}
Does the (topological or smooth) concordance group contain elements which are torsion of order greater than 2?
\end{quest}

In this section, we show that for any finite set $N$ of integers greater than 1, there exists $R\subset S^3$ such that for each $n\in N$, $\mathcal{C}_R$ has an element which is torsion of order $n$.

\begin{exmp} \label{ex:torsion}
Let $N=\{n_1,\dots,n_k\}$ be a set of positive integers greater than 1, and let $m=\text{lcm}(N)$.
Fix $R=\#^m(3_1)$, the connected sum of $m$ copies of the right-handed trefoil.
We will show that the relative concordance group $\mathcal{C}_R$ contains an element of order $n_i$ for each $n_i\in N$.
Let $d_i=m/n_i$, and define $K_i=\#^{d_i}(3_1)$.
Then $\#^{n_i}K_i=R$, so the order of $K_i$ in $\mathcal{C}_R$ divides $n_i$.
We show that the order of $K_i$ in $\mathcal{C}_R$ is exactly $n_i$.

Recall that the Alexander polynomial of the trefoil is $\Delta_{3_1}(t)=t^2-t+1=\lambda_{\pi/3}(t)$.
The associated Milnor signature $\sigma_{\pi/3}([3_1])$ has value $-1$, and $\sigma_{\pi/3}([-3_1])=1$. 
Note that this realizes a nontrivial invariant for Milnor's original collection of signatures $\{\sigma_\theta\}_{\theta\in S}$.
For any $k\in\N$, the Alexander polynomial of $\#^k(3_1)$ is $\Delta_{\#^k(3_1)}(t)=(\Delta_{3_1}(t))^k$, and $\sigma_{\pi/3}(\#^k(3_1))=-k$ by additivity.
As $\Delta_{\#^k(3_1)}(t)$ has the same roots as $\Delta_{3_1}(t)$, the signature $\sigma_{\pi/3}$ is the only nontrivial Milnor signature for $\#^k(3_1)$.
Analogous statements hold for $\#^k(-3_1)$.

Consider the surjective homomorphism \[\overline{\sigma}_{\pi/3}:\mathcal{AC}_R\twoheadrightarrow\Z/\langle\sigma_{\pi/3}([R_{n,1}]),\sigma_{\pi/3}([(-R)_{n,1}])\rangle_{n\in\Z}\] induced by the signature $\sigma_{\pi/3}:\mathcal{AC}\twoheadrightarrow\Z$.
Recall that $\sigma_{\pi/3}$ induces $\overline{\sigma}_{\pi/3}$ by Propositions~\ref{prop:halfdies} and~\ref{prop:quotient}.
To compute the signatures $\sigma_{\pi/3}([R_{n,1}])$ and $\sigma_{\pi/3}([(-R)_{n,1}])$, we recall a result of Kearton \cite{Kearton} which gives the precise relationship between the Milnor signatures of a knot and those of its satellites.
We restate this result in the special case of $(n,1)$-cables, where the nonzero signatures of $K_{n,1}$ take on the same values as those of $K$ (up to a sign), precisely at the corresponding $n$-th roots.

\begin{thm}[Kearton \cite{Kearton}] \label{thm:Kearton}
Let $K\subset S^3$, and let $0<\theta<\pi$ such that $e^{\theta i}$ is a root of $\Delta_{K_{n,1}}(t)$.
Then $\sigma_{\theta}(K_{n,1})=\pm\sigma_{n\theta}(K)$.
\end{thm}

By Theorem~\ref{thm:Kearton}, the codomain of $\overline{\sigma}_{\pi/3}$ is $\Z/m\Z$.
As $\sigma_{\pi/3}$ induces $\overline{\sigma}_{\pi/3}$, $\overline{\sigma}_{\pi/3}([K_i])=-d_i\in\Z/m\Z$.
As $K_i$ has finite order in $\mathcal{C}_R$, hence finite order in $\mathcal{AC}_R$, $n_i=m/d_i$ divides the order of $K_i$ in $\mathcal{AC}_R$, so $n_i$ divides the order of $K_i$ in $\mathcal{C}_R$. 
Thus, $n_i$ equals the order of $K_i$ in $\mathcal{C}_R$.
Observe that we obtain elements of any order dividing $m$ in the same way.
This construction applies to other choices of $R=\#^mK$, where $K$ has a nontrivial Milnor signature.

\end{exmp}


\bibliographystyle{plain}
\bibliography{references}


\end{document}